\documentclass[reqno,a4paper,12pt]{amsart}
\usepackage{amsmath,amsthm,amsfonts,amssymb, mathrsfs}
\usepackage{verbatim, graphicx, ifthen, enumitem}
\usepackage[T1]{fontenc}
\usepackage [applemac] {inputenc}
\usepackage{color}

\allowdisplaybreaks

\let\oldmarginpar\marginpar
\renewcommand\marginpar[1]{\-\oldmarginpar[\raggedleft\footnotesize #1]%
{\raggedright\footnotesize #1}}
\newtheorem{theorem}{Theorem}
\newtheorem{lemma}{Lemma}

\newtheorem{proposition}{Proposition}
\newtheorem{corollary}{Corollary}
\theoremstyle{definition}

\theoremstyle{remark}

\newcommand{\D}{\mathbb{D}}
\newcommand{\abs}[1]{|#1|}
\newcommand{\Abs}[1]{\left|#1\right|}
\newcommand{\Bigabs}[1]{\Big|#1\Big|}

\newcommand{\inner}[2]{\left\langle #1|#2 \right\rangle}

\newcommand{\N}{\mathbb{N}}

\newcommand{\C}{\mathbb{C}}

\def\N{\mathbb{N}}

\def\C{\mathbb{C}}

\def\1{\mathbf{1}}

\newcommand{\dif}{\mathrm{d}}
\newcommand{\e}{\mathrm{e}}
\newcommand{\im}{\mathrm{i}}
\newcommand{\norm}[1]{\|#1\|}

\newcommand{\Per}{\operatorname{Per}}
\newcommand{\Area}{\operatorname{Area}}

\renewcommand{\Re}{\operatorname{Re}}

\begin{document}

 \title[On a sharp estimate for Hankel operators  and Putnam's inequality]{On a sharp estimate for Hankel operators \\ and Putnam's inequality}

\author{Jan-Fredrik Olsen}
\address{Centre for Mathematical Sciences, Lund University, P.O. Box 118, SE-221 00 Lund, Sweden}
\email{janfreol@maths.lth.se}

\author{Maria Carmen Reguera}
\address{Departament de Ma\-te\-m\`a\-ti\-ques, Universitat Aut\`onoma de Bar\-ce\-lo\-na, Catalonia, Spain and School of Mathematics, University of Birmingham, Birmingham, United Kingdom.} 
\email{m.reguera@bham.ac.uk}

\begin{abstract} 
We obtain a sharp norm estimate for Hankel operators with anti-analytic symbol for weighted Bergman spaces. For the classical Bergman space, the estimate improves the
corresponding classical Putnam inequality for commutators of Toeplitz operators with analytic symbol by a factor of $1/2$, answering a recent conjecture by Bell, Ferguson and Lundberg.  As an application, this yields a new proof of the de St. Venant inequality, which relates the torsional rigidity 
of a domain with its area. 
\end{abstract}
  
 \maketitle

\section{Introduction}

Guided by applications to classical isoperimetric inequalities, we are concerned with finding sharp estimates for certain Hankel operators acting on the Bergman space $A^2(\D)$. More precisely, we are looking for sharp estimates in terms of the derivative of the symbol, which we assume to be anti-analytic.

 We recall that $A^2(\Omega)$ denotes the Bergman space on $\Omega$, i.e., the closed subspace of holomorphic functions in $L^2(\Omega, \dif A)$, where $\dif A = \dif x \dif y/\pi$. We let $E^2(\Omega)$ denote the Smirnov-Hardy space on $\Omega$. It is the closure of rational functions with poles outside of $\bar{\Omega}$ in the space $L^2(\partial \Omega, \abs{\dif z}/2\pi)$. When $\Omega = \D$, this is simply the classical Hardy space. 
	
These spaces are part of a scale of weighted Hilbert spaces of analytic functions on $\Omega$ which we denote by $A^2_\alpha(\Omega)$. See Section \ref{sec.putnam} for details. Here, we restrict ourselves to mentioning that the Bergman space corresponds to $\alpha=0$ while the Hardy space is obtained in the limit as $\alpha \rightarrow -1$. For convenience of notation, we   set $A_{-1}^2(\Omega) = E^2(\Omega)$ and let $L^2_\alpha$ denote the corresponding $L^2$ spaces.
	
	For each of these spaces, we denote by  $P : L^2_\alpha \rightarrow A^2_\alpha$ the orthogonal projection. For the $A^2(\D)$ this is the classical Bergman projection, while for $E^2(\D)$, this is   the Szeg\H o projection. 

The Hankel operator on $A^2_\alpha$, for some suitable symbol $\psi$, is defined as 
$$
H_{\psi}(f):= (I-P)(\psi f),  \qquad f\in A^2_\alpha,
$$
where $I$ denotes the identity map.
We will also need  the Toeplitz operator on $A^2_\alpha$, with symbol $\psi$, which is defined by
$$
T_{\psi}(f):= P(\psi f),  \qquad  f\in A^2_\alpha.
$$
We note that the Hankel operator sends $A^2_\alpha$ into its orthogonal complement in $L^2_\alpha$, 
and that $H_\psi + T_\psi  =  M_\psi$, the multiplication operator on $L^2_\alpha$ with symbol $\psi$. For
more background  on the theory of Hankel operators on Bergman spaces, we refer the reader to \cite{hedenmalm_korenblum_zhu2000}.
%
%

Our main result is as follows.
\begin{theorem}\label{t.maint}
Let $\alpha \geq -1$ and $\psi$ be a holomorphic function on $\D$ such that  $ \psi' \in A^2_\alpha(\D)$.  Then
\begin{equation}\label{e.maint}
\norm{H_{\overline{\psi}}}_{A^2_\alpha(\D)\rightarrow L^2_\alpha(\D)} \leq \frac{\norm{ \psi'  }_{A^2(\D)}}{\sqrt{2+\alpha}} .
\end{equation}
Moreover, this inequality is sharp.
\end{theorem}
The estimate in \eqref{e.maint} can be used to obtain a Putnam type inequality on spaces defined on simply connected domains $\Omega$. We recall that Putnam's inequality \cite{putnam1970} asserts the following: let $H$ be a Hilbert space and let $T:H\mapsto H$ be a bounded linear operator whose spectrum is denoted by $\sigma(T)$. If $T$ is hyponormal, that is, $\langle T^\ast T- TT^\ast f, f \rangle\geq 0$ for all $f\in H$, then 
\begin{equation}\label{e.putor}
\norm{T^\ast T- TT^\ast} \leq \frac{\text{Area}(\sigma(T))}{\pi}.
\end{equation}

As an application of Putnam's inequality to the Smirnov-Hardy space $E^2(\Omega)$, where we suppose that $\Omega$ is a reasonable domain, Khavinson \cite{khavinson1985} deduced the classical isoperimetric inequality. Indeed,     he obtained the   lower bound  
$$
\norm{T_{z}^\ast T_{z}- T_{z}T_{z}^\ast}_{E^2(\Omega)\rightarrow   E^2(\Omega)}\geq \frac{4\text{Area}(\Omega)^2}{\text{Per}(\Omega)^2},
$$
where $\text{Per}(\Omega)$ denotes  the perimeter of the region $\Omega$, and noted that when combined with Putnam's inequality, this lower bound immediately yields the isoperimetric inequality of the region $\Omega$. 

Bell, Ferguson and Lundberg \cite{bell_ferguson_lundberg2012} considered the problem of obtaining a lower bound  for the commutators of Toeplitz operators on the Bergman space $A^2(\Omega)$. In particular, they obtained the   inequality
\begin{equation} \label{belletal}
\norm{T_{z}^\ast T_{z}- T_{z}T_{z}^\ast}_{A^2(\Omega)\rightarrow A^2(\Omega) }\geq  \frac{\rho_\Omega}{\text{Area}(\Omega)}.
\end{equation}
Here, $\rho_\Omega$ is a constant from mechanics known as the torsional rigidity. It is geometric in the sense that it only depends on the shape of the domain $\Omega$, and can be said to measure the resistance an object with cross-section $\Omega$ has to rotation (see Section \ref{sec.rigidity} for a precise definition).

As was observed in \cite{bell_ferguson_lundberg2012}, when combined with Putnam's inequality, this yields 
$$
\rho(\Omega)\leq \frac{\text{Area}(\Omega)^{2}}{\pi}.
$$
This is close to the classical de St. Venant inequality
$$
\rho(\Omega)\leq \frac{{\Area}(\Omega)^{2}}{2\pi}.
$$
Consequently, Bell, Ferguson and Lundberg   conjectured that in this setting it should be possible to improve Putnam's inequality by a factor of $1/2$. This conjecture is answered in the positive by Theorem \ref{t.maint}.  To be precise, we obtain the following corollary.
\begin{corollary}\label{c.putnam}
	Let $\alpha \geq -1$, $\Omega\subset\C$ be a simply connected domain and $\psi$ an analytic function. Then $$\norm{T_\psi^\ast T_\psi  - T_\psi T_\psi^\ast}_{A^2_\alpha(\Omega) \rightarrow A^2_\alpha(\Omega)} \leq \frac{\norm{\psi'}_{A^2(\Omega)}^2}{2+\alpha}.$$
\end{corollary}
We explain how to deduce this corollary in Section \ref{sec.putnam}. 

Hence, by choosing $\psi=z$, we immediately get the following improvement of Putnam's inequality for the shift operators $T=T_z$ on $\Omega$. Indeed, in 
this case
\begin{equation*}
	 \norm{T^\ast T - TT^\ast }_{A^2(\Omega) \rightarrow A^2(\Omega)}  \leq \frac{\norm{1}_{A^2(\Omega)}^2}{2} = \frac{\text{Area}(\Omega)}{2\pi}.
\end{equation*}
(Recall that   the planar measure $\dif A$ is normalized on $\D$, whence the factor $1/\pi$.)
By the arguments of Bell, Ferguson and Lundberg, a new proof of the de St. Venant inequality now follows.

We also record the following consequence of Theorem \ref{t.maint}. It is simply the observation that for analytic functions $h$ on  $\Omega$, then $u = H_{\bar{z}}h$ satisfies $\bar{\partial} u = h$. Here $\bar{\partial} = (\partial_x + \im \partial_y)/2$ is the classical d-bar operator.
\begin{corollary}
	Let $h$ be an  analytic function on a simply connected domain $\Omega$. Then there exists a function $u \in L^2(\Omega,\dif A_\alpha)$ such that $\bar{\partial} u = h$ with
	\begin{equation*}
	 	\norm{u}_{L^2(\Omega, \dif A_\alpha)}^2 \leq \frac{\Area(\Omega)}{(2+\alpha)\pi } \norm{h}_{A^2_\alpha(\Omega)}^2.
	\end{equation*}
\end{corollary}

Finally, a natural question is whether one can interpolate the lower bounds of Khavinson, one the one hand, and Bell, Ferguson and Lundberg on the other. Unfortunately we are not able to do this, but we make several remarks on the question of a lower bound in Section \ref{sec.rigidity}. In particular, we obtain an inequality
\begin{equation*}
		\norm{[T_z^\ast,T_z ]}_{A^2_\alpha(\Omega)} \geq \frac{\rho_{\Omega,\alpha}}{\int_\Omega w_\alpha \, \dif x \dif y},
\end{equation*}
where $\rho_{\Omega,\alpha}$ is a quantity analogous to the torsional rigidity associated with the spaces $L^2_\alpha(\Omega)$. We observe, by an example, that for the Smirnov-Hardy spaces $E^2(\Omega)$, this lower bound is in general larger than the one of Khavinson.

The structure of the paper is as follows. In Section \ref{sec.putnam}, the relation between the norm of the commutator and the norm of a Hankel operator on the disc $\D$ is shown, as well as the proof of Corollary \ref{c.putnam} using Theorem \ref{t.maint}. Section \ref{sec.maint} is dedicated to the proof of the Theorem \ref{t.maint} in the unweighted Bergman case, $\alpha=0$, while Section \ref{sec.weightedt} explains how to extend the proof to the general case. Finally, Section \ref{sec.rigidity} is dedicated to the proof of the lower bounds for the commutators of Toeplitz operators with symbol $\bar{z}$ in the weighted Bergman spaces in terms of torsional rigidities.


\section{The relation to Hankel operators  and a proof of Corollary \ref{c.putnam}} \label{sec.putnam}
%


We begin by explicitly defining the spaces $A_\alpha^2(\Omega)$ mentioned above. Let $\Omega$ be a simply connected domain in $\C$ and suppose that $\phi : \Omega \rightarrow \D$ is the corresponding univalent and analytic Riemann mapping. For $\alpha > -1$, we define $A^2_\alpha(\Omega)$ to be the closure of holomorphic functions on $\Omega$ in the norm given by
\begin{equation*}
	\norm{f}_{\alpha}^2  = \int_\Omega \abs{f(z)}^2 \dif A_\alpha(z),
\end{equation*}
where $$\dif A_\alpha(z) = (1+\alpha) \frac{(1- \abs{\phi(z)}^2)^\alpha}{\abs{\phi'(z)}^\alpha} \frac{\dif x \dif y}{\pi}.$$ 

For $\alpha=-1$, we set $A^2_\alpha(\Omega)$ to be the Smirnov-Hardy space $E^2(\Omega)$, defined to be the closure of the rational functions with poles outside of $\bar{\Omega}$ in the norm given by
\begin{equation*}
	\norm{f}_{E^2(\Omega)}^2  = \int_{\partial \Omega} \abs{f(z)}^2 \frac{ \abs{\dif z} }{2\pi}.
\end{equation*}
A standard reference for Smirnov-Hardy spaces is \cite[Chapter 10]{duren1970}. 

We remark that with these definitions, it follows that
\begin{equation*}
	\lim_{\alpha \rightarrow -1^+} \norm{f}_\alpha^2  =   \norm{f}^2_{E^2(\Omega)}.
\end{equation*}
Moreover, recall that we set $L^2_\alpha(\Omega) = L^2(\Omega, \dif A_\alpha)$ for $\alpha>-1$ and $L^2_\alpha(\Omega) = L^2(\partial \Omega, \abs{\dif z}/2\pi)$ for $\alpha=-1$.

For $\alpha>-1$ the orthogonal projection from $L^2_\alpha$ to $A^2_\alpha$ is given by
\begin{equation}
\label{e.rep}
	Pg(z)= \int_{\Omega} g(w) \frac{\overline{\phi'(w)^{1+\alpha/2}}\phi'(z)^{1+\alpha/2}}{(1-\phi(z)\overline{\phi(w)})^{2+\alpha}} \dif A_\alpha(w),
\end{equation}
while for $\alpha=-1$ it is given by
\begin{equation*}
	Pg(z) = \int_{\partial \Omega} g(w) \frac{\overline{\phi'(w)^{1/2}}  \phi'(z)^{1/2}}{1-\phi(z)\overline{\phi(w)}} \frac{\abs{\dif w}}{2\pi}.
\end{equation*}

With this, one can now write down explicit expressions for both the Toeplitz operators $T_\psi : f \longmapsto P(\psi f)$ and the Hankel operators $H_\psi : f \longmapsto (I-P)(\psi f)$ on these spaces.

The following lemma is useful.
\begin{lemma} \label{lemma}
	Let $\alpha \geq -1$ and $\Omega$ a simply connected domain in $\C$. If $\phi : \Omega \rightarrow \D$ is a Riemann mapping, then
	\begin{equation*}
		\norm{H_{\psi}}_{A^2_\alpha(\Omega)} =  \norm{H_{ {\psi \circ \phi^{-1}}}}_{A^2_\alpha(\D)},
	\end{equation*}
	where the   operators $H_\psi$ and $H_{\psi \circ \phi^{-1}}$ are Hankel   for the spaces  $A^2_\alpha(\Omega)$ and $A^2_\alpha(\D)$, respectively.
\end{lemma}
\begin{proof}
We verify the formula for $\alpha=0$, the other cases follow an identical argument. By the above, we have the formula
\begin{align*}
	H_{\bar{\psi}}f(z) 
	= \int_{\Omega}   \frac{\phi'(z) \overline{\phi'(w)}}{\big( 1    -   \phi(z) \overline{\phi(w)}\big)^2} \big(\overline{\psi(z)}-  \overline{\psi(w})\big) f(w) \dif A (w).
\end{align*}
We now make the   change of variables  $\zeta = \phi(z)$ and $\omega = \phi(w)$. Passing the integration from $\Omega$ to $\D$, we get the 
Jacobian $ \abs{(\phi^{-1})'(\omega)}^2$. Thus, the above expression is equal to
\begin{multline*}
	\int_{\D}   \frac{\phi' ( \phi^{-1}(\zeta) ) \overline{\phi'(\phi^{-1}(\omega))}}{\big( 1    -  \zeta \overline{\omega}\big)^2} \big(\overline{\psi\circ \phi^{-1}(\zeta)}-  \overline{\psi \circ \phi^{-1}(\omega)}\big) \big(f\circ \phi^{-1}(\omega)\big)  \abs{(\phi^{-1})'(\omega)}^2\dif A (\omega) \\
	=
		\int_{\D}   \frac{ \overline{\psi\circ \phi^{-1}(\zeta)}-  \overline{\psi \circ \phi^{-1}(\omega)}}{\big( 1    -  z \overline{\omega}\big)^2}  \;
		\frac{(\phi^{-1})'(\omega)}{(\phi^{-1})'(\zeta) }
   \; \big( f\circ \phi^{-1}(\omega)  \big) \dif A (\omega).
\end{multline*}
Changing notations, this yields the formula
\begin{align*}
	(\phi^{-1})'(\zeta) \big((H_{\overline{\psi}} f) \circ \phi^{-1}\big)(\zeta) 
   =
	 H_{\overline{\psi\circ\phi^{-1}}}\; g(\zeta),
\end{align*}
where $g(\zeta) = (\phi^{-1})'(\omega)   \; (f\circ \phi^{-1})(\omega)$ and $H_{\overline{\psi\circ\phi^{-1}}}$ is a   Hankel operator on the disk. Taking norms, it is clear that
\begin{equation*}
	\norm{H_{\overline{\psi}}}_{A^2(\Omega)\rightarrow L^2(\Omega)} 
	= \norm{H_{\overline{\psi\circ\phi^{-1}}}}_{A^2(\D)\rightarrow L^2(\D)}.
\end{equation*}
\end{proof}

\subsection*{Proof of Corollary \ref{c.putnam}} Let  $T = T_\psi$, where $\psi$ is an   analytic map on $\Omega$. The strategy is first to relate the commutator to a Hankel operator on $A^2_\alpha(\Omega)$ and then  to  pass  to the unit disk using Lemma \ref{lemma}. There, we  apply  Theorem \ref{t.maint} which yields the result. Indeed, by a straight-forward computation, we get
\begin{align*}
\|T^{*}T-TT^{*}\|_{A_\alpha^2(\Omega)\rightarrow A_\alpha^2(\Omega)} & =
\sup_{\substack{h\in A_\alpha^{2}(\Omega)\\ \|h\|_{A_\alpha^{2}}=1}}\langle (T^{*}T-TT^{*})h, h\rangle\\
& =\sup_{\substack{h\in A_\alpha^{2}(\Omega)\\  \|h\|_{A_\alpha^{2}}=1}} \Big(  \|Th\|_{A_\alpha^{2}(\Omega)}^{2}-  \|T^{*}h\|_{A_\alpha^{2}(\Omega)}^{2} \Big) \\
&= \sup_{\substack{h\in A_\alpha^{2}(\Omega)\\ \|h\|_{A_\alpha^{2}}=1}}  \Big( \|\psi h\|_{A_\alpha^{2}(\Omega)}^{2}- \|P(\overline{\psi}h)\|_{A_\alpha^{2}(\Omega)}^{2} \Big) \\
&= \sup_{\substack{h\in A_\alpha^{2}(\Omega)\\ \|h\|_{A_\alpha^{2}}=1}}  \Big( \|\overline{\psi}h\|_{L_\alpha^{2}(\Omega)}^{2}- \|P(\overline{\psi}h)\|_{A_\alpha^{2}(\Omega)}^{2} \Big) \\ 
&=   \norm{H_{\overline{\psi}}}_{A_\alpha^2(\Omega) \rightarrow L_\alpha^2(\Omega)}^2.
\end{align*}
Applying   Lemma \ref{lemma} followed by Theorem \ref{t.maint}, we get
\begin{equation*}
	 \norm{H_{\overline{\psi\circ\phi^{-1}}}}_{A_\alpha^2(\D)\rightarrow L_\alpha^2(\D)}^2 \leq \frac{\norm{(\psi \circ \phi^{-1})'}_{A^2(\D)}^2}{2 + \alpha} =  \frac{\norm{\psi' }_{A^2(\Omega)}^2}{2 + \alpha}.
\end{equation*}
This proves Corollary \ref{c.putnam}. \qed

\section{Proof of Theorem \ref{t.maint} in the case $\alpha=0$} \label{sec.maint}

In this section we prove our main theorem for $\alpha=0$. This gives essentially the idea of the proof in the general case, without the notation getting too much in the way. In the next section, we explain in detail how to extend this proof to the general case.

 The strategy of the proof is as follows. For $f \in A^2(\D)$, write $f(z) = \sum_{n \geq 0} a_n z^n$, and set $\psi(z) = \sum_{k\geq 1} c_k z^k$ (note that we can assume $\psi(0)=0$ without loss of generality). We then   express the function $H_{\overline{\psi}}f$   in terms  of these Taylor coefficients, and   obtain the desired  norm estimate by working directly with the coefficients. Essentially, the only  inequality we use is $ab \leq (1/2)(a^2 + b^2)$.

As a first step, we observe that 
\begin{align*}
	P(\bar{z}^k z^n) = \int_\D \frac{\bar{w}^k w^n }{(1-\bar{w}z)^{2}} \dif A(w) &= \sum_{\ell\geq0} (\ell+1)z^\ell \int_\D \bar{w}^{k+\ell} w^n \dif A(w)  \\
	&=\left\{ \begin{aligned}  \frac{n-k+1}{n+1} z^{n-k} & \quad \text{ if } 0 \leq k \leq n  \\  0  & \quad \text{ if } n < k \end{aligned} \right..
\end{align*}
From the previous computation, for $n\geq 1$, it follows that
\begin{align*}
	P(\overline{\psi} z^n) = \sum_{k\geq 1} \bar{c}_k P  (\overline{z}^k z^n)  &= \sum_{k=1}^n \frac{n-k+1}{n+1} \bar{c}_k  z^{n-k} \\
	&= \sum_{k=0}^{n-1}  \frac{k+1}{n+1}   \bar{c}_{n-k} z^k,
\end{align*}
whence we get the   expression
\begin{align*}
	 H(\overline{\psi}f)(z) &=  \overline{\psi}(z) f(z) - P(\overline{\psi}f)(z)\\
	 &= \sum_{\ell \geq 1} \sum_{n \geq 0} \bar{c}_\ell a_n \bar{z}^\ell z^n-  \sum_{n\geq 1} \sum_{k=0}^{n-1} \frac{k+1}{n+1}  a_n \bar{c}_{n-k} z^k.
\end{align*}
Our next objective is to integrate with respect to   $\int_0^{2\pi} \dif \theta/\pi$. Before we do this, we rewrite this expression in order
to more easily use the orthogonal structure. 
First, we rewrite the above expression as
\begin{equation} \label{some equation}
	  \sum_{l\geq 1} a_0 \bar{c}_l \bar{z}^l  + \sum_{n\geq 1} a_n \Big( \underbrace{  \sum_{l\geq 1} \bar{c}_l \bar{z}^l z^n - \sum_{0 \leq k \leq n-1}\frac{k+1}{n+1} \bar{c}_{n-k} z^k }_{(*)}\Big).
\end{equation}
Now,
\begin{align*}
	(*) &=  \sum_{1 \leq \ell \leq n} \bar{c}_\ell \abs{z}^{2l} z^{n-\ell}  + \sum_{\ell \geq n+1} \bar{c}_\ell \abs{z}^{2n} \bar{z}^{\ell-n} - \sum_{0 \leq k \leq n-1} \frac{k+1}{n+1} \bar{c}_{n-k} z^k \\
	&= \sum_{1 \leq k \leq n} \bar{c}_{n-k} \abs{z}^{2(n-k)} z^{k}  + \sum_{k \geq 1} \bar{c}_{k+n} \abs{z}^{2n} \bar{z}^{k} - \sum_{0 \leq k \leq n-1}  \frac{k+1}{n+1} \bar{c}_{n-k} z^k.
\end{align*}
Plugging this back in to equation \eqref{some equation}, we get
\begin{multline*}
	 \sum_{k\geq 1} \bar{z}^k \Big( a_0 \bar{c}_k   +  \sum_{n\geq 1 } a_n \bar{c}_{k+n} \abs{z}^{2n} \Big) 
	 + \sum_{n\geq 1} a_n \sum_{0 \leq k \leq n-1}\bar{c}_{n-k} z^k \Big( \abs{z}^{2(n-k)} - \frac{k+1}{n+1}  \Big) \\ = 
	 \sum_{k\geq 1} \bar{z}^k \Big( \sum_{n\geq 0 } a_n \bar{c}_{k+n} \abs{z}^{2n} \Big) 
	 +  \sum_{k \geq 0}  z^k \sum_{n\geq k+1} a_n\bar{c}_{n-k} \Big( \abs{z}^{2(n-k)} - \frac{k+1}{n+1}  \Big).
\end{multline*}
Taking the modulus squared of this at $z=r\e^{\im \theta}$ and integrating with respect to $\int_0^{2\pi} \dif \theta/\pi$ yields
\begin{equation*}
	\underbrace{2 \sum_{k\geq 1} r^{2k}  \Bigabs{ \sum_{n\geq 0 } a_n \bar{c}_{k+n} r^{2n} }^2}_{(I)}
	 + \underbrace{ 2\sum_{k \geq 0}  r^{2k} \Bigabs{ \sum_{n\geq k+1} a_n\bar{c}_{n-k} \Big( r^{2(n-k)} - \frac{k+1}{n+1}  \Big)}^2}_{(II)}.	 
\end{equation*}
Before integrating with respect to the radial part of  the measure $\dif A = r \dif r \dif \theta/\pi$, we multiply out  and rewrite these expressions slightly:
\begin{align*}
	(I) =  2  \sum_{\underset{k\geq 1}{n,m\geq 0} } a_n \bar{a}_m c_{k+m} \bar{c}_{k+n} r^{2n+2m+2k},
\end{align*}
and
\begin{multline*}
	(II)  = 2\sum_{k \geq 0}   \sum_{n,m\geq k+1} a_n \bar{a}_m c_{m-k}\bar{c}_{n-k} \Big( r^{2n + 2m -2k}  - r^{2n} \frac{k+1 }{m+1} \\ - r^{2m} \frac{k+1 }{n+1} +  r^{2k} \frac{(k+1)^2 }{(n+1)(m+1)}  \Big).
\end{multline*}
Integration against $\int_0^1 r \dif r$ now yields
\begin{align*}
	(I) &=   \sum_{\underset{k\geq 1}{n,m\geq 0} } \frac{a_n \bar{a}_m c_{k+m} \bar{c}_{k+n} }{n+m+k+1}, \\
	(II)  &= \sum_{k \geq 0}   \sum_{n,m\geq k+1} \frac{a_n \bar{a}_m c_{m-k}\bar{c}_{n-k}  (m-k)(n-k)}{(n+1)(m+1)(n+m-k+1)}.
\end{align*}
At this point, we make a change of coefficients by setting $a_n = b_{n+1} (n+1)$. With this notation,
\begin{align*}
	(I) &=   \sum_{\underset{k\geq 1}{n,m\geq 0} }  b_{n+1}    \bar{b}_{m+1}     c_{k+m}   \bar{c}_{k+n}      \frac{(n+1)(m+1)}{n+m+k+1}, \\
	(II)  &= \sum_{k \geq 0}   \sum_{n,m\geq k+1} b_{n+1} \bar{b}_{m+1} c_{m-k}\bar{c}_{n-k}  \frac{ (m-k)(n-k)}{ n+m-k+1}.
\end{align*}
Adjusting the indices of summation slightly,  
\begin{align*}
	(I) &=   \sum_{\underset{k\geq 0}{n,m\geq 1} }  b_{n}    \bar{b}_{m}     c_{k+m}   \bar{c}_{k+n}      \frac{nm}{n+m+k}, \\
	(II)  &= \sum_{\underset{k\geq 1}{n,m\geq 1} }    b_{n+k} \bar{b}_{m+k} c_{m}\bar{c}_{n}  \frac{ mn}{ n+m+k}.
\end{align*}
By the symmetry in $m,n$ we may interpret each term as being half that of its real part, so that the inequality 
$2\Re(ab) \leq \abs{a}^2 + \abs{b}^2$ applied to each of these expression,   yields
\begin{align*}
	(I) &\leq    \sum_{\underset{k\geq 0}{n,m\geq 1} }  \Big(\abs{b_{n} c_{k+m}}^2 +    \abs{ {b}_{m}    {c}_{k+n}}^2  \Big)     \frac{nm}{2(n+m+k)}  \\
	& =   \sum_{\underset{k\geq 0}{n,m\geq 1} }   \abs{b_{n} c_{k+m}}^2     \frac{nm}{n+m+k} =: (I_*),  \\
	(II)  &\leq  \sum_{\underset{k\geq 1}{n,m\geq 1} }  \Big(\abs{b_{n+k} c_m}^2 + \abs{  {b}_{m+k}  {c}_{n} }^2 \Big) \frac{ mn}{ 2(n+m+k)}\\
	& = \sum_{\underset{k\geq 1}{n,m\geq 1} }  \abs{b_{n+k} c_m}^2  \frac{ mn}{ n+m+k} =: (II_*).
\end{align*}
Here, we again used the symmetry in $m,n$. The next step is to isolate all unique pairs of indices. We change the order of summation as follows,
\begin{align*}
	(I_*) &=    \sum_{ n\geq 1 } \sum_{m\geq 1}  \sum_{k\geq m} \abs{b_{n} c_{k} }^2  \frac{nm}{n+k}  =    \sum_{ n,k\geq 1 } \abs{b_{n} c_{k} }^2  \sum_{m=1}^{k} \frac{nm}{n+k}\\
	& =   \sum_{ n,k\geq 1 } \abs{b_{n} c_{k} }^2    \frac{nk(k+1)}{2(n+k)},\\
	(II_*) &=   \sum_{ m\geq 1} \sum_{n\geq 1} \sum_{k\geq n+1}   \abs{b_{k}c_m}^2 \frac{nm}{m+k}  = \sum_{ \underset{k\geq 2}{m\geq 1}}     \abs{b_{k}c_m}^2  \sum_{n=1}^{k-1} \frac{nm}{m+k}\\ 
	&= \sum_{ \underset{k\geq 2}{m\geq 1}}     \abs{b_{k}c_m}^2   \frac{mk(k-1)}{2(m+k)}.
\end{align*}
We notice that we can add the index $k=1$ to the second sum without changing its value. Finally, taking this into account, we add these terms together to get
\begin{align*}
	(I_*) + (II_*)  &=  \sum_{n,m\geq 1} \abs{b_n}^2 \abs{c_m}^2  \left(  \frac{nm(m+1)}{2(n+m)} +  \frac{mn(n-1)}{2(m+n) }\right) \\
	&=\sum_{n,m\geq 1} \abs{b_n}^2 \abs{c_m}^2   \frac{nm}{2}.
\end{align*}
Replacing $a_n = b_{n+1}(n+1)$, we now see that the right-hand side exactly equals
\begin{equation*}
	\frac{1}{2} \sum_{\underset{m\geq 1}{n\geq 0}} \abs{b_n}^2 \abs{c_m}^2   \frac{m}{n+1} =   \frac{1}{2} \Big( \sum_{n\geq 0} \frac{\abs{a_n}^2}{n+1}\Big) \Big( \sum_{m\geq 1} \abs{c_m}^2 m \Big) = \frac{1}{2} \norm{f}_{A^2(\D)}^2 \norm{\psi'}_{A^2(\D)}^2,
\end{equation*}
which was to be shown.

\section{Proof of Theorem \ref{t.maint} for $\alpha \geq -1$} \label{sec.weightedt}
 
We now explain how to prove Theorem \ref{t.maint} for the spaces $A^2_\alpha(\D)$. Actually, we restrict us to the case $\alpha >-1$, leaving the simplest case, $\alpha=-1$, to the interested reader.
%
%

We begin by  recalling some additional facts about these spaces. These can be checked by consulting the introductory chapter of the reference  \cite{hedenmalm_korenblum_zhu2000}. First, the   monomials $z^n$ form an orthogonal base for the space 
and have norm
\begin{equation*}
	\norm{z^n}_\alpha^2 = \frac{n! \Gamma(\alpha+2)}{\Gamma(n + \alpha + 2)},
\end{equation*}
where $\Gamma(z)$ is the usual gamma-function satisfying the functional relation $\Gamma(z+1)=z\Gamma(z)$. To keep the notation simple, we  introduce the coefficients
\begin{equation*}
	D_n^\alpha =   \frac{n! \Gamma(\alpha+2)}{\Gamma(n + \alpha + 2)}.
\end{equation*}
In particular, we have
\begin{equation*}
	D_n^0 =  \frac{n! \Gamma(2)}{\Gamma(n+2)} = \frac{n!}{(n+1)!} = \frac{1}{n+1} \qquad \text{and} \qquad D_n^{-1} = \frac{n! \Gamma(1)}{\Gamma(n+1)} = 1.
\end{equation*}
Furthermore, it is well-known that 
\begin{equation*}
	 \sum_{n \geq 0} \frac{(z\overline{w})^n}{D_n^\alpha}  = \frac{1}{(1 - \overline{w}z)^{2+\alpha}}.
\end{equation*}
From this, and the orthogonality of monomials,   the   orthogonal projection from $L^2(\D,\dif A_\alpha)$ to $A_\alpha(\D)$ is given by 
\begin{equation*}
	f \longmapsto P_\alpha f(z) = \int_\D \frac{f(w)}{(1-\bar{w}z)^{2+\alpha}} \dif A_\alpha(w),
\end{equation*}
and so  we have an explicit expression for the associated Hankel operator
\begin{equation*}
	H_{\overline{\psi}} f(z) = \overline{\psi}(z) f - P_\alpha(\overline{\psi}f)(z).
\end{equation*}
Repeating the same type of arguments as above, we get an expression in terms of the Taylor coefficients of $f = \sum_{n\geq 0} a_n z^n$ and $\psi = \sum_{k \geq 1} c_k z^k$.
\begin{align*}
	 H(\bar{\psi}f)(z) 
	 &= \sum_{\ell \geq 1} \sum_{n \geq 0} \bar{c}_\ell a_n \bar{z}^\ell z^n-  \sum_{n\geq 1} \sum_{k=0}^{n-1}  \frac{D_n^\alpha}{D_k^\alpha} a_n \bar{c}_{n-k} z^k.
\end{align*}
As before, the next steps consist of  reorganizing these sums and taking its norm.  Doing this, we get
\begin{multline} \label{weightednorm}
	\norm{H_{\overline{\psi}}f}_\alpha^2 =   
		  \underbrace{ \sum_{\underset{k\geq 1}{n,m\geq 0} } a_n \bar{a}_m c_{k+m} \bar{c}_{k+n} D_{n+m+k}^\alpha}_{(I)}
	\\+ \underbrace{\sum_{k \geq 0}   \sum_{n,m\geq k+1} a_n \bar{a}_m c_{m-k}\bar{c}_{n-k} \Big(D_{n+m-k}^\alpha -  \frac{D_n^\alpha D_m^\alpha}{D_k^\alpha}   \Big)}_{(II)}.
\end{multline}
Here, the labeling into $(I)$ and $(II)$ exactly match the previous case. After this, we make the change of coefficients $D_n^\alpha a_n = b_{n+1}$. Some algebra then yields
\begin{multline*}
	(I) + (II) =      \sum_{\underset{k\geq 0}{n,m\geq 1} } b_{n} \bar{b}_{m} c_{k+m} \bar{c}_{k+n} \frac{D_{n+m+k-1}^\alpha}{D_{n-1}^\alpha D_{m-1}^\alpha}\\
	 +   \sum_{\underset{k \geq 1}{n,m\geq 1}}  b_{n+k} \bar{b}_{m+k} c_{m}\bar{c}_{n} \Big( \frac{D_{n+m+k-1}^\alpha}{D_{n+k-1}^\alpha D_{m+k-1}^\alpha} -  \frac{1}{D_{k-1}^\alpha}   \Big).
\end{multline*}
Using the symmetry in the indices $m,n$, we apply the inequality $2\Re ab \leq \abs{a}^2+\abs{b}^2$ in the same way as in the unweighted case, to obtain
\begin{multline*}
	(I) + (II) \leq  
		 \sum_{\underset{k\geq 0}{n,m\geq 1} }  \abs{b_{n} c_{k+m} }^2  \frac{D_{n+m+k-1}^\alpha}{D_{n-1}^\alpha D_{m-1}^\alpha}  \\
	   +   \sum_{\underset{k \geq 1}{n,m\geq 1}}   \abs{b_{n+k}c_m}^2   \Big( \frac{D_{n+m+k-1}^\alpha}{D_{n+k-1}^\alpha D_{m+k-1}^\alpha} -  \frac{1}{D_{k-1}^\alpha}   \Big).
\end{multline*}
The next step is to isolate all unique pairs of indices. Changing the indices and the order of summation, in exactly the same way as before, we get that the above expression is equal to
\begin{equation} \label{final countdown}
	  \sum_{ n,m\geq 1 } \abs{b_{n} c_{m} }^2  \sum_{\ell=0}^{m-1}\frac{D_{n+m-1}^\alpha}{D_{n-1}^\alpha D_{\ell}^\alpha}  
	 +  \sum_{ \underset{n\geq 2}{m\geq 1}}     \abs{b_{n}c_m}^2  \sum_{\ell=m}^{n+m-2}  \Big( \frac{D_{n+m-1}^\alpha}{D_{n-1}^\alpha D_{\ell}^\alpha} -  \frac{1}{D_{\ell-m}^\alpha}   \Big).
\end{equation}
The proof of the desired inequality is complete once we show that this expression is smaller than 
\begin{equation*}
	\frac{\norm{f}_\alpha^2 \norm{\psi'}_\alpha^2}{2+\alpha} =  \frac{1}{2+\alpha} \sum_{\underset{n\geq 0}{m\geq 1}} \frac{\abs{c_m}^2 \abs{b_n}^2 m}{D_{n-1}^\alpha}.
\end{equation*}
This follows exactly by applying the following  lemma to the sums in the expression \eqref{final countdown}.
\begin{lemma} Suppose  $v \in \N$, then
\begin{equation*}
	 \sum_{\ell=0}^{v} \frac{1}{D^\alpha_\ell}  =  \frac{v+1}{2+\alpha} \frac{1}{D^\alpha_{v+1}}.
\end{equation*}
\end{lemma}
\begin{proof}
	The relation holds for $v=0$. Indeed,
	\begin{equation*}
		D^\alpha_{1} = \frac{1! \Gamma(\alpha+2)}{\Gamma(3+\alpha)} = \frac{\Gamma(2+\alpha)}{(2+\alpha)\Gamma(2+\alpha)} = \frac{1}{2+\alpha} D^\alpha_0.
	\end{equation*}
	The induction step is easily verified from the    observation
	\begin{equation*}
		D^\alpha_v = \frac{v! \Gamma(2+\alpha)}{\Gamma(v+\alpha+2)}  = \frac{v}{v+\alpha+1} D^\alpha_{v-1}.
	\end{equation*}
\end{proof}
Finally, we note that the sharpness is seen from formula \eqref{weightednorm} by choosing $\Omega=\D$, $\psi = z$ and $f \equiv 1$. This ends the proof.

\section{Torsional rigidity and lower bounds} \label{sec.rigidity}

Let $\Omega$ be a bounded and simply connected domain in $\C$. In this section we consider the problem of 
obtaining a lower bound for the commutator of $T_\psi$ on $A^2_\alpha(\Omega)$ in the case that $\psi = z$. By the proof of Corollary \ref{c.putnam}, it is the same as  obtaining a lower bound for
 $\norm{H_{\bar{z}}}_{A^2_\alpha(\Omega) \rightarrow L^2_\alpha(\Omega)}^2$.

As above, for $\alpha>-1$,  we denote by $A_\alpha^2(\Omega)$ the weighted Bergman spaces, while we set $A^2_{-1} := E^2(\Omega)$, the
Smirnov-Hardy space on $\Omega$.

With this notation, for $\alpha\geq -1$, and in light of the computation in Corollary \ref{c.putnam}, we get
\begin{align*}
	\norm{[T_z, T_z^\ast]}_\alpha  =  \sup_{ {h \in A_\alpha^2  }}  \frac{ \norm{H_{\bar{z}} h }^2_\alpha}{\norm{h}_\alpha^2}
	&=  \sup_{ {h \in A_\alpha^2  }}  \frac{  \norm{\bar{z}h}^2_\alpha - \norm{P(\bar{z}h)}^2_\alpha }{\norm{h}^2_\alpha} \\
	&=  \sup_{ {h \in A_\alpha^2  }} \frac{\operatorname{dist}(\bar{z}h, A^2_\alpha)}{\norm{h}^2_\alpha} \\
	&=  \sup_{ {h \in A_\alpha^2  }} \inf_{f \in A^2_\alpha} \frac{ \norm{\bar{z} h - f}^2_\alpha}{\norm{h}^2_\alpha} \\
	&=
	\sup_{ {h \in A_\alpha^2  }} \inf_{f \in A^2_\alpha} \sup_{ {g \in L^2_\alpha   }} \frac{ \abs{ \inner{\bar{z}h - f}{g}_\alpha}^2}{\norm{g}^2_\alpha \norm{h}_\alpha^2}. 
\end{align*}


We turn our attention to the formulation in terms of the supremums. By restricting the innermost supremum to $(A^2_\alpha)^\perp$, this yields the apparent inequality
\begin{equation*}
	\norm{H_{\bar{z}} h }^2_\alpha  \geq   \sup_{ {g \in (A^2_\alpha)^\perp  }}  \frac{\abs{ \inner{\bar{z}h }{g}_\alpha}^2}{\norm{g}_\alpha^2}.
\end{equation*}
However, one can check that this extremum is attained by  $g = H_{\bar{z} } h $ in $(A^2_\alpha)^\perp$. Hence, we get the identity
\begin{equation} \label{variational problem}
	\norm{H_{\bar{z}} h }^2_\alpha  =   \sup_{ {g \in (A^2_\alpha)^\perp}}  \frac{\abs{ \inner{\bar{z} h }{g}_\alpha}^2}{\norm{g}_\alpha^2}.
\end{equation}

It is now easy to explain how Khavinson \cite{khavinson1985} used essentially this relation for $\alpha=-1$ to deduce the isoperimetric inequality. The key observation is that by Cauchy's integral formula, the function $g =    \overline{\dif z}/\abs{\dif z}$ is seen to be in $E^2(\Omega)^\perp$. Indeed, combined  with the choice $h=1$ and Green's formula, this yields
\begin{align*}
	\norm{[T^\ast,T]}_\alpha \geq \frac{\norm{H_{\bar{z}} 1 }^2_\alpha }{\norm{1}_\alpha}  &\geq   \frac{\abs{ \inner{\bar{z} }{g}_\alpha}^2}{\norm{1}_\alpha \norm{g}_\alpha^2} \\
	&=  \frac{1}{\Per (\Omega)^2} \Abs{ \int_{\partial \Omega} \bar{z} \, \dif z }^2 \\
	&=  \frac{1}{\Per (\Omega)^2} \Abs{2\im  \int_{\Omega}  (\bar{\partial} \bar{z}) \, \dif z }^2 	=  \frac{4 \Area(\Omega)^2 }{\Per (\Omega)^2}.
\end{align*}
Applying either Putnam or our Theorem \ref{t.maint}, the isoperimetric inequality 
\begin{equation*}
	  \Area(\Omega)  \leq \frac{\Per (\Omega)^2}{4\pi}
\end{equation*}
immediately  follows.

As stated above, Bell, Ferguson and Lundberg mimicked more or less the same approach to get a lower bound for the case $\alpha=0$. More specifically, they used the relation \eqref{variational problem} and noted that   $A^2(\Omega)^\perp$ may be expressed as 
\begin{equation*}
	\operatorname{clos} \{ \partial \psi : \psi \in C^\infty_0(\bar{\Omega}) \}.
\end{equation*}
Here $\psi \in C^\infty_0(\Omega)$ if it is smooth on $\Omega$, continuous up to and including on the boundary where it also vanishes. That this holds can be seen by using Green's formula in combination with the Hahn-Banach theorem. So, for the Bergman space, they observed that
\begin{equation} \label{dualitycalc}
	\begin{aligned}
	\norm{[T^\ast,T]}  \geq \frac{\norm{H_{\bar{z}} 1 }^2 }{\norm{1} } 
	&= 
	\sup_{ {g \in (A^2)^\perp}}  \frac{\abs{ \inner{\bar{z}  }{g}}^2}{\norm{g} ^2 \norm{1} ^2} \\
	&= 
	\sup_{ \psi \in C^\infty_0(\bar{\Omega}) }  \frac{\abs{ \inner{\bar{z}  }{\partial \psi} }^2}{\norm{\partial \psi} ^2 \norm{1} ^2} \\
	&= 
	\sup_{ \psi \in C^\infty_0(\bar{\Omega}) }  \frac{4\abs{ \inner{1  }{  \psi} }^2}{\norm{\nabla \psi} ^2 \norm{1} ^2} 
	= 
	\frac{\rho_\Omega}{\Area(\Omega)}.
	\end{aligned}
\end{equation}
Here, the quantity $\rho_\Omega$ on the left-hand side is what is known in elasticity theory as the torsional rigidity of the domain $\Omega$. In fact, classically
\begin{equation} \label{torsional rigidity 1}
	\rho_\Omega = 	\sup_{ \psi \in C^\infty_0(\bar{\Omega}) }  \frac{4 \Big(  \int_\Omega \psi \dif x \dif y \Big)^2}{ \int_\Omega (\partial_x \psi^2 + \partial_y \psi^2 ) \dif x \dif y},
\end{equation}
 where the supremum is only over real valued functions $\psi$. 

As we noted in the introduction, the torsional rigidity is a constant from mechanics which  quantifies the resistance to rotation of a cylindrical object, imagined as being  perpendicular to the complex plane, with cross-section equal to $\Omega$ at all heights.  We refer the reader to  \cite[p. 2]{polya_szego1951} for a more accurate physical description.  Mathematically, this quantity has several equivalent definitions. See, e.g. \cite[p.87--89]{polya_szego1951} for a discussion of this. Observe, in particular, that $$\norm{H_{\bar{z}}1}^2_{A^2(\Omega)} = \rho_\Omega.$$

In order to take a closer look at the case $\alpha \in (-1,0)$, we look at yet another way of expressing the torsional rigidity of a simply connected domain $\Omega$  (see, e.g., \cite{polya_szego1951}). Namely    if $v$ is the
solution of the Dirichlet problem
\begin{equation*}
	\left\{ \begin{aligned}  \Delta v = -2 \\ v |_{\partial \Omega} = 0 \end{aligned}\right.,
\end{equation*}
then
\begin{equation} \label{torsional rigidity 2}
	\rho_\Omega = 2 \int_\Omega v \, \dif x \dif y. 
\end{equation}
To see that this is equivalent to \eqref{torsional rigidity 1}, first observe that, assuming all function are real valued, by the Cauchy-Schwarz inequality we have
\begin{align*}
	 \int_\Omega 2 \psi \dif x \dif y =  - \int_\Omega \Delta v \, \psi \, \dif x \dif y  &= \int_\Omega \nabla v \cdot \nabla \psi \, \dif x \dif y \\
	& \leq \bigg(  \int_\Omega \abs{\nabla v}^2 \dif x \dif y  \int_\Omega \abs{\nabla \psi}^2 \dif x \dif y \bigg)^{1/2}.
\end{align*}
Since 
\begin{equation*}
	2 \int_\Omega v  \dif x \dif y  = -\int_\Omega v \Delta v \dif x \dif y =  \int_\Omega  \abs{\nabla v}^2 \dif x \dif y,
\end{equation*}
this implies that $2 \int v \, \dif x \dif y \geq \rho_\Omega$.
On the other hand, simply setting $\psi = v$ in \eqref{torsional rigidity 1}, one obtains the identity \eqref{torsional rigidity 2}.

Elaborating on the above discussion, we observe that if we set $u  = -2 \partial v$, then $\bar{\partial} u = 1$. That is $u$ solves the same
d-bar problem as $\bar{z} - P_+(\bar{z})$. Indeed, these functions are the same. More generally, if $h$ is analytic on $\Omega$ and $v_h$ solves
\begin{equation*}
	\left\{ \begin{aligned}  \Delta v_h = h \\ v |_{\partial \Omega} = 0 \end{aligned}\right.,
\end{equation*}
then $u_h = 4 \partial v_h$ solves $\bar{\partial} u_h = v$. Since $u_h$ is seen, by Green's formula, to be orthogonal to $A^2(\Omega)$, it follows
that 
\begin{equation*}
	u_h = H_{\bar{z}} h.
\end{equation*}
 
Returning to the case $\alpha \in (-1,0)$, we pose the following Dirichlet problem
\begin{equation*}
	\left\{ \begin{aligned} \Delta_\alpha v = -2 \\ v |_{\partial \Omega} = 0 \end{aligned}\right.,
\end{equation*}
where 
\begin{equation*}
 \Delta_\alpha=  4 \bar{\partial} \frac{1}{w_\alpha} \partial   
\end{equation*}
and $w_\alpha$ is the weight so that $\dif A_\alpha(z) = w_\alpha(z) \dif A(z)$, the measure defining the weighted Bergman spaces $A^2_\alpha(\Omega)$. With this,
we define a weighted  torsional rigidity as follows:
\begin{equation*}
	\rho_{\Omega,\alpha} = 2 \int_\Omega {v} \, \dif x \dif y.
\end{equation*}

By repeating the steps of \eqref{dualitycalc} above, using now that, by Green's formula, $u := -2w_\alpha^{-1} \partial v$ is in $(A^2_\alpha)^\perp$ and it is also a solution to the d-bar problem $\bar{\partial} u = 1$, therefore $u=H_{\bar{z}}1$ , it is now straight-forward to obtain the following result.
\begin{proposition}
	Let $\Omega$ be a simply connected domain in $\C$. For $\alpha \in (-1,0)$, we have
	\begin{equation*}
		\norm{[T_z^\ast,T_z ]}_{A^2_\alpha(\Omega)} \geq \frac{\rho_{\Omega,\alpha}}{\int_\Omega w_\alpha \, \dif x \dif y}.
	\end{equation*}
\end{proposition}

The above result does not immediately make sense for $\alpha=-1$, since we are no longer dealing with
a norm given by an area measure, but rather one given by the arc length measure of the boundary. Even though, in this case, we 
could give a similar formulation in terms of laplacians, it is more convenient to note that we actually have
\begin{equation*}
	\rho_{\Omega, \alpha} = \norm{H_{\bar{z}} 1 }^2_\alpha  =   \sup_{ {g \in (A^2_\alpha)^\perp}}  \frac{\abs{ \inner{\bar{z} }{g}_\alpha}^2}{\norm{g}_\alpha^2},
\end{equation*}
which is a more straight-forward relation that holds for all $\alpha \in [-1,0]$.  With this, the above proposition also extends to the Smirnov-Hardy space.

The question   now  is whether $\norm{H_{\bar{z}}1}_\alpha^2/\norm{1}^2_\alpha$ is identical to Khavinson's lower bound? While, the answer is yes for  discs, in general, it is  no, as is seen by the following example.

\subsection*{Example 1} We  consider the map 
$\phi(z) = (1+z)^2 - 1$. This is a univalent conformal and analytic map from   $\D$ onto   the domain $\Omega$ shown in Figure \ref{figure omega}. 
\begin{figure}
\includegraphics[scale=0.5]{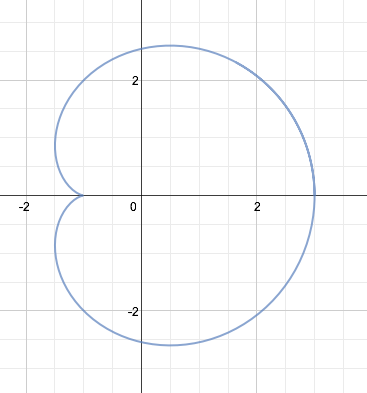}
\caption{The domain $\Omega$ of Example 1.}
\label{figure omega}
\end{figure}
It is a basic calculus exercise to calculate
\begin{align*}
	\operatorname{Per}(\Omega) &= \int_{\partial \D} \abs{\phi'(z)} \abs{ \dif z} \\
	&= 2 \int_0^{2\pi} \abs{1+ \e^{\im \theta}} \dif \theta \\
	&= 2^{3/2} \int_0^{2\pi} \sqrt{ 1 - \sin \theta} \dif \theta
	= 16, 
\end{align*}
and
\begin{equation*}
	\operatorname{Area}(\Omega) = \int_{\D} \abs{\phi'(z)}^2 \dif x \dif y = 6\pi.
\end{equation*}

The Khavinson lower bound   is now
\begin{equation*}
	\norm{[T,T^\ast]} \geq 4 \frac{\operatorname{Area}(\Omega)^2}{\operatorname{Per}(\Omega)^2} = 4 \frac{(6\pi)^2}{16^2} =  \frac{9\pi^2}{16}.
\end{equation*}

As may be calculated explicitly, the torsional rigidity lower bound becomes 
\begin{equation*}
	\frac{\norm{H_{\bar{z}}1}_{E^2(\Omega)}^2}{\norm{1}^2_{E^2(\Omega)}}  =  \frac{\norm{H_{\bar{\phi}} \sqrt{\phi'(z)} }^2_{L^2(\dif \D)} }{\norm{1}^2_{E(\Omega)}} = \frac{2\pi}{\operatorname{Per}(\Omega)}   \sum_{k\geq 1} \Bigabs{\sum_{ \ell \geq k}  \gamma_{\ell-k} \bar{c}_\ell}^2,
\end{equation*}
where $\phi(z) = \sum_{k\geq 1} c_kz^k$ and   $\sqrt{\phi'(z)} = \sum_{k \geq 0} \gamma_k z^k$. Calculating this, we get
\begin{equation*}
	\frac{\norm{H_{\bar{z}}1}_{E^2(\Omega)}^2}{\norm{1}^2_{E^2(\Omega)}} = \frac{1}{16} 2 \pi \cdot \frac{29}{2}  = \frac{29 \pi}{16}.
\end{equation*}
Since
\begin{equation*}
 \frac{29 \pi}{16} > \frac{9\pi^2}{16},
\end{equation*}
this shows that the two lower bounds are in general different.

\section*{Acknowledgements} Both authors would like to thank professor Joaquim Ortega-Cerd\`a for helpful conversations, as well as professors Kristian Seip and Yurii Lyubarskii for inviting them to the Centre for Advanced study in Oslo during the spring 2013, where most of the work on this paper was done. 

\bibliographystyle{amsplain}
\def\cprime{$'$} \def\cprime{$'$} \def\cprime{$'$} \def\cprime{$'$}
\providecommand{\bysame}{\leavevmode\hbox to3em{\hrulefill}\thinspace}
\providecommand{\MR}{\relax\ifhmode\unskip\space\fi MR }
\providecommand{\MRhref}[2]{%
  \href{http://www.ams.org/mathscinet-getitem?mr=#1}{#2}
}
\providecommand{\href}[2]{#2}

\end{document}